\documentclass[12pt, a4paper, english]{amsart}
\usepackage{amsmath} 
\usepackage{amsthm} 
\usepackage{amssymb} 
\usepackage{amscd} 
\usepackage{array} 
\usepackage[mathscr]{eucal} 
\usepackage{hyperref} 
\usepackage{caption} %
\usepackage{graphics,graphicx} 
\usepackage{tikz,tikz-cd} 
\usepackage{enumerate} 
\usepackage[margin=2.8cm]{geometry}
\usepackage{verbatim}
\usepackage{xcolor}
\makeatletter
\@namedef{subjclassname@2020}{\textup{2020} Mathematics Subject Classification}
\makeatother
\usetikzlibrary{graphs,graphs.standard,calc}

\hypersetup{
    colorlinks = true,
    linkbordercolor = {white},
    linkcolor = {purple},
    anchorcolor = {black},
    citecolor = {blue},
    filecolor = {cyan},
    menucolor = {red},
    runcolor = {cyan},
    urlcolor = {magenta}
}

\everydisplay{\let\binom\mybinom}

\newcommand{\stirling}[2]{\biggl[\genfrac{}{}{0pt}{}{#1}{#2}\biggr]}
\newcommand{\tstirling}[2]{\left[\genfrac{}{}{0pt}{}{#1}{#2}\right]}

\usetikzlibrary{automata}

\newtheoremstyle{teoremas}
{10pt}
{10pt}
{\itshape}
{}
{\bfseries}
{}
{.5em}
{}

\theoremstyle{teoremas}
\newtheorem{teo}{Theorem}[section]

\newtheorem{cor}[teo]{Corollary}

\newtheorem{lema}[teo]{Lemma}
\newtheorem{conj}[teo]{Conjecture}

\newtheorem{prop}[teo]{Proposition}

\theoremstyle{definition}
\newtheorem{defi}[teo]{Definition}
\newtheorem{ej}[teo]{Example}
\newtheorem{obs}[teo]{Remark}

\DeclareMathOperator{\ehr}{ehr}

\DeclareMathOperator{\rk}{rk}

\newcommand{\M}{\mathsf{M}}

\newcommand{\U}{\mathsf{U}}
\newcommand{\T}{\mathsf{T}}

\title{Matroids are not Ehrhart positive}

\author[L. Ferroni]{Luis Ferroni}
\thanks{This paper was written during the author's third year as a PhD Student at the University of Bologna, funded by the Marie Sk{\l}odowska-Curie fellowship INdAM-DP-COFUND-2015, Grant Number 713485.}

\address{Department of Mathematics, KTH Royal Institute of Technology, Stockholm, Sweden} 

\email{ferroni@kth.se}


\allowdisplaybreaks

\begin{document}

\begin{abstract}
    In this article we disprove the conjectures asserting the positivity of the coefficients of the Ehrhart polynomial of matroid polytopes by De Loera, Haws and K\"oppe (2007) and of generalized permutohedra by Castillo and Liu (2015). We prove constructively that for every $n\geq 19$ there exist connected matroids on $n$ elements that are not Ehrhart positive. Also, we prove that for every $k\geq 3$ there exist connected matroids of rank $k$ that are not Ehrhart positive. Our proofs rely on our previous results on the geometric interpretation of the operation of circuit-hyperplane relaxation and our formulas for the Ehrhart polynomials of hypersimplices and minimal matroids. This allows us to give a precise expression for the Ehrhart polynomials of all sparse paving matroids, a class of matroids which is conjectured to be predominant and which contains the counterexamples arising from our construction. \\
    
\end{abstract}

\maketitle

\section{Introduction}\label{sec:one}

A fundamental result by Ehrhart \cite{ehrhart} shows that when $\mathscr{P}\subseteq \mathbb{R}^n$ is a lattice polytope, the function $\ehr(\mathscr{P}, -)$ counting the number of lattice points in the integral dilations of $\mathscr{P}$, namely
    \[ \ehr(\mathscr{P}, t) = \# (t\mathscr{P}\cap\mathbb{Z}^n),\]
is a polynomial in the variable $t$ of degree $d = \dim \mathscr{P}$. It is customary to call $\ehr(\mathscr{P},t)$ the \emph{Ehrhart polynomial} of $\mathscr{P}$. 

This result has motivated the study of these polynomials as both combinatorial and arithmetic invariants of lattice polytopes. A basic result in Ehrhart theory is that if $\dim \mathscr{P} = d$, then
    \begin{equation}\label{basica}\ehr(\mathscr{P}, t) = \operatorname{vol}(\mathscr{P}) t^d + \frac{1}{2}\operatorname{vol}(\partial\mathscr{P}) t^{d-1} + \cdots + 1,\end{equation}
where $\operatorname{vol}$ is the function that associates to a lattice polytope its relative volume.

Observe that equation \eqref{basica} reveals that the coefficients of degrees $d$, $d-1$ and $0$ of the Ehrhart polynomial of a polytope $\mathscr{P}$ of dimension $d$ are always positive. However, the coefficients accompanying the terms of degrees $1,\ldots, d-2$ are not as well-understood. Although following McMullen \cite{mcmullen} it is possible to derive general formulas for each of the coefficients of $\ehr(\mathscr{P},t)$, they are quite complicated.

For the members of some families of basic polytopes such as regular simplices, hypercubes, cross-polytopes and hypersimplices, there exist explicit formulas that permit to compute all of their Ehrhart coefficients. Many of these examples are addressed in the books by Beck and Robins \cite{beck} and Beck and Sanyal \cite{becksanyal}.

Some examples in dimension $3$, such as the so-called ``Reeve's tetrahedron'' \cite[Example 3.22]{beck}, exhibit that sometimes the coefficients of the Ehrhart polynomial of a polytope can be negative. In fact, Hibi et al. \cite{hibi-negative} showed that \emph{all} the coefficients of degrees $1,\ldots,d-2$ can be negative simultaneously. Even when one restricts to the family of $0/1$-polytopes, i.e., polytopes whose vertices have all of their coordinates equal to $0$ or $1$, it is possible to find examples that have negative Ehrhart coefficients; see \cite{liuorder} by Liu and Tsuchiya for an example.

When $\mathscr{P}$ is a lattice polytope such that $\ehr(\mathscr{P},t)$ has positive coefficients, we say that $\mathscr{P}$ is \emph{Ehrhart positive}. A main reference about Ehrhart positivity is Liu's survey \cite{liu}.
 
One of the main open problems in this framework was a conjecture posed in 2007 by De Loera, Haws and K\"oppe \cite[Conjecture 2]{deloera}.

\begin{conj}[\cite{deloera}]\label{deloera}
    Let $\M$ be a matroid and $\mathscr{P}(\M)$ be its base polytope. Then $\mathscr{P}(\M)$ is Ehrhart positive.
\end{conj}

Generalizing the notion of matroid to that of polymatroid allows to construct a broader class of polytopes, which correspond to the base polytopes of polymatroids. Polymatroids were introduced by Edmonds and Rota \cite{edmonds-rota} and have proven useful to address many important problems in combinatorial optimization and algebraic combinatorics. In \cite{postnikov} Postnikov studied certain polytopes arising as deformations of permutohedra, which he named ``generalized permutohedra''. Later it was realised that generalized permutohedra and base polytopes of polymatroids are (up to a translation) the same objects.

In 2015 \cite[Conjecture 1.2]{castilloliu} Castillo and Liu posed the following conjecture, which is an extension of Conjecture \ref{deloera} to all generalized permutohedra.

\begin{conj}[\cite{castilloliu}]\label{castilloliu}
    If $\mathscr{P}$ is an integral generalized permutohedron, then $\mathscr{P}$ is Ehrhart positive.
\end{conj}

There was much evidence to support these two conjectures. In what follows we give a brief overview of some known results backing up these two assertions. In \cite{deloera} De Loera, Haws and K\"oppe computed the Ehrhart polynomials of an extensive list of base and independence polytopes of matroids and polymatroids and in all cases the coefficients were observed to be positive. Also, they were able to prove that rank $2$ uniform matroids were Ehrhart positive; this result was later extended by Ferroni in \cite{ferroni1}, proving that uniform matroids of all ranks are Ehrhart positive.

\begin{teo}[\cite{ferroni1}]
    If $\M$ is a uniform matroid, then its base polytope $\mathscr{P}(\M)$ is Ehrhart positive.
\end{teo}

In \cite{ferroni3} Ferroni used this result to show that independence polytopes of uniform matroids are Ehrhart positive as well.
In \cite{postnikov} Postnikov gave a proof of the fact that all the members of a quite large family of generalized permutohedra, which he called ``$\mathcal{Y}$-generalized permutohedra'', were Ehrhart positive. This family can be described as the set of all (positive) Minkowski sums of dilations of standard simplices that satisfy a certain modularity property; however, in spite of being very general, it does not contain the class of matroid polytopes as a subfamily, as is pointed out in \cite{ardilabenedetti}.

Also, in the same paper in which they conjectured the Ehrhart positivity of generalized permutohedra, Castillo and Liu proved that the Ehrhart coefficients of degree $d-2$ and $d-3$ of a generalized permutohedron are always positive. In \cite{castillo2020todd} they proved that the linear coefficient is always positive. The following is a summary of their results.

\begin{teo}[\cite{castilloliu}, \cite{castillo2020todd}]\label{coef}
    Let $\mathscr{P}$ be an integral generalized permutohedron of dimension $d\geq 3$, and let $\ehr(\mathscr{P},t)$ denote its Ehrhart polynomial. Then:
    \begin{itemize}
        \item $[t^1] \ehr(\mathscr{P},t)$ is positive.
        \item $[t^{d-2}] \ehr(\mathscr{P},t)$ is positive.
        \item $[t^{d-3}] \ehr(\mathscr{P},t)$ is positive.
    \end{itemize}
\end{teo}

A different proof of the positivity of the linear coefficient was independently found by Jochemko and Ravichandran in \cite{jochemko}. It is worth mentioning that the positivity of the linear term was a particularly important clue, since most of the known examples of polytopes with at least one negative Ehrhart coefficient do have a negative linear term.

Observe that from Castillo and Liu's result, it follows that a generalized permutohedron of dimension $d\leq 5$ is automatically Ehrhart positive. For example, for $d=5$, the coefficients of degree $1$, $2$ and $3$ are covered by Theorem \ref{coef} and the coefficients of degree $0$, $4$ and $5$ are covered by \eqref{basica}. They were able to extend this result to all generalized permutohedra of dimension $d\leq 6$.

However, in spite of all this evidence, in the present article we will disprove Conjecture \ref{deloera} and thus also Conjecture \ref{castilloliu}. A related conjecture that will also be disproved is the lower bound part of \cite[Conjecture 1.5]{ferroni2}.
We will construct explicit examples of matroids that fail to have positive Ehrhart coefficients. Thus, we definitely answer on the negative the Ehrhart positivity question for matroids and generalized permutohedra.

\subsection*{Outline and main results}

In Section \ref{sec:two} we describe the basic properties of matroids, with special emphasis on the notion of circuit-hyperplane relaxation and its relation with the class of sparse paving matroids. In Section \ref{sec:three} we analyze the geometric counterpart of the results and definitions described in Section \ref{sec:two}, we also discuss the equivalence between generalized permutohedra and base polytopes of polymatroids. In Section \ref{sec:four} we review the Ehrhart theory of uniform and minimal matroids and we give an explicit formula for the Ehrhart polynomial of all sparse paving matroids. In Section \ref{sec:five} we use a result from coding theory and the formula for sparse paving matroids to construct a counterexample to Conjecture \ref{deloera}. 

\begin{teo}
    There exists a connected matroid $\M$ of cardinality $20$, rank $9$, having $159562$ bases, such that the base polytope $\mathscr{P}(\M)$ is not Ehrhart positive.
\end{teo}

Moreover, we give an explicit description of this matroid. In Section \ref{sec:six} we extend this construction to higher dimensions. 

\begin{teo}\label{greaterthan9}
    For every $n\geq 19$ there exists a connected matroid $\M$ on $n$ elements such that $\mathscr{P}(\M)$ is not Ehrhart positive.
\end{teo}

We also prove that we can find counterexamples to Conjecture \ref{deloera} for all ranks greater than or equal to $3$.

\begin{teo}\label{main3}
    For every $k\geq 3$ and $n$ sufficiently large there exists a connected matroid $\M$ of rank $k$ and cardinality $n$ such that $\mathscr{P}(\M)$ is not Ehrhart positive.
\end{teo}

We describe precise lower bounds for $n$ (depending on $k$) that show that for example if $n\geq 3589$ there is always a matroid of rank $3$ and cardinality $n$ that is not Ehrhart positive.

In Section \ref{sec:seven} we discuss the rank $2$ case and sketch how to prove that all sparse paving matroids of rank $2$ are Ehrhart positive. We finish the paper in Section \ref{sec:eight}, where we discuss some relevant facts and some extra examples.

\section{Matroids as a combinatorial structure}\label{sec:two}

In this section we review several notions on matroid theory that we will need in the sequel. Our primary sources are \cite{oxley} by Oxley and \cite{schrijver} by Schrijver. 

\subsection{Matroids}

We start by summarizing the basic terminology and concepts on matroid theory.

\begin{defi}
    A \emph{matroid} $\M$ is a pair $(E,\mathscr{B})$ where $E$ is finite set and $\mathscr{B}$ is family of subsets of $E$, i.e. $\mathscr{B}\subseteq 2^E$ that satisfies the following two conditions.
    \begin{enumerate}[(a)]
        \item $\mathscr{B}\neq \varnothing$.
        \item For each $B_1\neq B_2$ members of $\mathscr{B}$ and $a\in B_1\smallsetminus B_2$, there exists an element $b\in B_2\smallsetminus B_1$ such that $(B_1\smallsetminus \{a\})\cup \{b\}\in \mathscr{B}$.
    \end{enumerate}
\end{defi}

We usually refer to condition (b) as the \emph{basis-exchange-property} and call the members of $\mathscr{B}$ the \emph{bases} of $\M$. One of the most basic corollaries of this definition of matroids is that all the bases of a matroid have the same cardinality. 

One of the classical examples of matroids is that of \emph{uniform matroids}. Throughout this article we will denote by $\U_{k,n}$ the uniform matroid of rank $k$ and $n$ elements. Concretely, $\U_{k,n}$ is defined by $E=\{1,\ldots, n\}$ and $\mathscr{B} = \{ B\subseteq E: |B| = k\}$. 

A basic fact on matroid theory, which is motivated from notions coming from graph theory and linear algebra, is that matroids admit a notion of \emph{duality}. 

\begin{prop}
    Let $\M=(E,\mathscr{B})$ be a matroid. Then the family \[\mathscr{B}^* = \{E\smallsetminus B: B\in \mathscr{B}\}\] is the set of bases of a matroid on $E$. We denote this matroid by $\M^*$ and call it the \emph{dual} of $\M$.
\end{prop}

It is not difficult to see that $(\M^*)^*=\M$, so that the operation described in the preceding result is in fact an involution, and the term ``dual'' is justified. As a basic example, it can be seen directly using the definitions that $\U_{k,n}^* = \U_{n-k,n}$.

There are several concepts about matroids that we will use and refer to repeatedly throughout this article.

\begin{defi}
    Let $\M=(E,\mathscr{B})$ be a matroid.
    \begin{itemize}
        \item If $I\subseteq E$ is contained in some $B\in \mathscr{B}$, we say that $I$ is \emph{independent}. When a subset of $E$ is not independent, we say it is \emph{dependent}.
        \item If $C\subseteq E$ is dependent but every proper subset of $C$ is independent, we say that $C$ is a \emph{circuit}.
        \item For every $A\subseteq E$ we define its \emph{rank}, $\rk(A)$, by:
            \[ \rk(A) = \max_{B\in \mathscr{B}} |A\cap B|.\]
        We say that the rank of $\M$ is just $\rk(E)$.
        \item If $F\subseteq E$ is a subset such that for every $e\notin F$ it is $\rk(F\cup \{e\}) > \rk(F)$, we say that $F$ is a \emph{flat}. 
        \item If $H$ is a flat and $\rk(H) = \rk(E)-1$, we say that $H$ is a \emph{hyperplane}.
    \end{itemize}
\end{defi}

All of the objects that we have just defined satisfy nice properties which in turn can be used to give alternative definitions of matroids; we refer to \cite[Chapter 1]{oxley} for a thorough account of many of such properties. One that we will use is the \emph{submodular inequality} for the rank function. This states that if $\M=(E,\mathscr{B})$ is a matroid, then
    \[ \rk(A_1) + \rk(A_2) \geq \rk(A_1\cup A_2) + \rk(A_1\cap A_2),\]
for every $A_1,A_2\subseteq E$. 

\begin{obs}\label{circuitcohyperplane}
    A basic result in matroid theory states that whenever $C$ is a circuit of a matroid $\M=(E,\mathscr{B})$, then $E\smallsetminus C$ is a hyperplane of $\M^*$. Conversely, the complement of a hyperplane $H$ of $\M$ is a circuit of $\M^*$. In other words, the circuits of $\M$ and the hyperplanes of $\M^*$ correspond bijectively. See also \cite[Proposition 2.1.6]{oxley}.
\end{obs}

Another important operation in matroid theory is the so-called \emph{direct sum} of matroids. If $\M_1=(E_1,\mathscr{B}_1)$ and $\M_2=(E_2,\mathscr{B}_2)$ are matroids on disjoint ground sets $E_1$ and $E_2$, their direct sum is defined as the matroid $\M_1\oplus\M_2$ that has ground set $E_1\sqcup E_2$ and set of bases $\mathscr{B} = \{B_1\sqcup B_2:B_1\in\mathscr{B}_1, B_2\in\mathscr{B}_2\}$.

A matroid $\M$ is said to be \emph{connected} if for every pair of distinct elements of the ground set, there exists a circuit containing both of them. This is equivalent for a matroid to be indecomposable, in the sense that it is not a direct sum of two or more matroids. Every matroid can be uniquely decomposed as a direct sum $\M = \M_1 \oplus \cdots \oplus \M_s$ of connected matroids, which are called the \emph{connected components} of $\M$.

\subsection{Relaxations and sparse paving matroids}

When one has a matroid $\M=(E,\mathscr{B})$ it is natural to ask under which conditions one can enlarge the set of bases $\mathscr{B}$ and obtain a new matroid. Assume that we want to add exactly one extra member to the set of bases of $\M$ and produce a new matroid (i.e. we want our new set of bases to satisfy the basis-exchange-property). The following result provides a scenario in which this is possible.\footnote{In fact, due to a result by Truemper in \cite{truemper} this is essentially the only possible way of enlarging the set of bases by one extra element.}

\begin{prop}
	Let $\M=(E,\mathscr{B})$ be a matroid that has a subset $H$ that is at the same time a circuit and a hyperplane. Let $\widetilde{\mathscr{B}}=\mathscr{B}\cup\{H\}$. Then $\widetilde{\mathscr{B}}$ is the set of bases of a matroid $\widetilde{\M}$ on $E$.
\end{prop} 

\begin{proof}
    See \cite[Proposition 1.5.14]{oxley}.
\end{proof}

A subset $H$ that is at the same time a circuit and a hyperplane will be referred to as a \emph{circuit-hyperplane}. The operation of declaring a circuit-hyperplane to be a basis is known in the literature by the name of \emph{relaxation}. Many famous matroids arise as a result of applying this operation to another matroid. For example the ``Non-Pappus matroid'' is the result of relaxing a circuit-hyperplane on the ``Pappus matroid'', and analogously the ``Non-Fano matroid'' can be obtained by a relaxation of the ``Fano matroid'' (for some other examples see \cite{oxley}).

\begin{ej}
    The graph in Figure \ref{wheel} is often called the \emph{wheel graph} of length $7$ and is denoted by $\mathsf{W}_7$.
    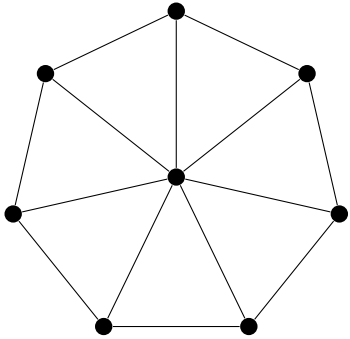
\begin{figure}[h]
	    \centering
        \begin{tikzpicture}  
		[scale=1.7,auto=center,every node/.style={circle, fill=black, inner sep=2.3pt}] 
		\tikzstyle{edges} = [thick];
		
        \graph  [empty nodes, clockwise, radius=1em,
        n=9, p=0.3] 
            { subgraph C_n [n=7,m=3,clockwise,radius=2.2cm,name=A]-- mid};
            \foreach \i [count=\xi from 2]  in {1,2,...,7}{
            \node[scale=0.1,fill=black, inner sep=2.3pt] at (A \i){\xi}; }
    \end{tikzpicture}\caption{The wheel graph $\mathsf{W}_7$}\label{wheel}
    \end{figure}
    One can construct a matroid from this graph by taking its set of edges as the ground set, with bases given by the spanning trees of the graph. The resulting matroid has cardinality $14$ and rank $7$. Notice that the graph $\mathsf{W}_7$ has a cycle of length $7$ (which in Figure \ref{wheel} consists of taking all the edges on the outer cycle); in the matroid this corresponds to a circuit of size $7$. Also, notice that this circuit is a flat of rank $6$, i.e., a hyperplane. In other words, it is a circuit-hyperplane and in particular it is possible to relax it. The matroid that one obtains with this relaxation is usually called the \emph{whirl matroid} $\mathsf{W}^7$.
\end{ej}

Let us describe a class of matroids that is intimately related with the circuit-hyperplane relaxation.

\begin{defi}
    Let $\M$ be a matroid of rank $k$. We say that $\M$ is \emph{paving} if every circuit of $\M$ has cardinality at least $k$. We say that $\M$ is \emph{sparse paving} if both $\M$ and its dual $\M^*$ are paving.
\end{defi}

Observe that if a matroid $\M$ of rank $k$ is paving, then its circuits must be all of size $k$ or $k+1$. Also, since according to Remark \ref{circuitcohyperplane} the hyperplanes of $\M$ are exactly the complements of circuits of $\M^*$, then if $\M^*$ is paving what we have is that all the hyperplanes of $\M$ have size exactly $k$ or $k-1$. 

If $\M$ is sparse paving, when one picks a circuit $C$ of length $k$, since its rank is $k-1$, it must be contained in a hyperplane $H$. Thus, $k=|C|\leq |H|\in \{k-1,k\}$. Hence, the only possibility is $|H|=k$, and therefore $C=H$. In particular $C$ is a hyperplane. Conversely, any hyperplane of size $k$ of a sparse paving matroid is a circuit.

\begin{lema}\label{sparsepavingdefi}
    A matroid $\M$ of rank $k$ is sparse paving if and only if every subset of cardinality $k$ is either a basis or a circuit-hyperplane.
\end{lema}

\begin{proof}
    Observe that if a matroid is such that every subset of cardinality $k$ is either a basis or a circuit-hyperplane, then it automatically is sparse paving. This is because the existence of a circuit of size less than $k$ is ruled out. Such a circuit can be completed to a set of cardinality $k$ which will fail to be a basis and a circuit-hyperplane. Analogous considerations avoid the possibility of the existence of a hyperplane of size greater than $k$.
    
    For the other implication, assume $\M$ is sparse paving and pick a subset $A$ of cardinality $k$ that is not a basis. Hence, we have that $A$ is dependent, and thus contains a circuit $C$. Since $\M$ is sparse paving we have that $k\leq |C| \leq |A| = k$, and since $C\subseteq A$, it follows that $C=A$ and hence $A$ is a circuit. Since $A$ has cardinality $k$, by the considerations prior to the statement of the Lemma, it follows that $A$ is a hyperplane.
\end{proof}

It follows from the above result that uniform matroids are sparse paving and, moreover, that every sparse paving matroid can be relaxed until obtaining a uniform matroid. This is because after relaxing one circuit-hyperplane, the remaining circuit-hyperplanes are still circuit-hyperplanes of the new matroid. 

\begin{obs}\label{predominance}
    In \cite{mayhew} Mayhew, Newman, Whittle and Welsh conjectured that \emph{asymptotically all} matroids are sparse paving. To be more precise, if we denote by $\operatorname{mat}(n)$ the number of \emph{labeled} matroids on $n$ elements (i.e. the number of matroids with ground set $\{1,\ldots,n\}$)\footnote{In \cite{mayhew} it is also conjectured that if we look at the number of matroids on $n$ elements \emph{up to isomorphism}, then the proportion of sparse paving matroids among them tends to 1 when $n\to\infty$ as well.} and $\operatorname{sp}(n)$ the number of sparse paving matroids among them, then
    
        \[\lim_{n\to\infty} \frac{\operatorname{sp}(n)}{\operatorname{mat}(n)} = 1.\]
    
    There is some evidence supporting that assertion. In fact, in \cite{pendavingh} Pendavingh and van der Pol proved that 
        \[\lim_{n\to\infty} \frac{\log\operatorname{sp}(n)}{\log \operatorname{mat}(n)} = 1.\]
\end{obs}

\begin{obs}
    If $n\geq 5$ and $2\leq k \leq n - 2$, then a sparse paving matroid of rank $k$ and cardinality $n$ is connected. To see this, a strategy consists of first determining explicitly all the disconnected paving matroids (see for instance \cite[Proposition 2.1]{oxley-paving}). Imposing that $\M$ and $\M^*$ are disconnected and paving yields that there is only one disconnected sparse paving matroid of rank and corank greater than $1$, namely $\U_{1,2}\oplus\U_{1,2}$, which is a matroid of rank $2$ and cardinality $4$.
\end{obs}

\section{Matroid polytopes, polymatroids and generalized permutohedra}\label{sec:three}

In this section we will review some definitions and fundamental facts about matroid polytopes, polymatroids and generalized permutohedra.

\subsection{Matroid polytopes} 

One of the many features of matroids is that they can be studied using tools from different areas of mathematics. In particular, starting from a matroid one can construct a polytope.

\begin{defi}
    Let $\M=(E,\mathscr{B})$ be a matroid on $E=\{1,\ldots,n\}$. For each $1\leq i \leq n$, let us call $e_i$ the $i$-th canonical vector in $\mathbb{R}^n$. For every $A\subseteq E$ define
        \[ e_A = \sum_{i \in A} e_i \in \mathbb{R}^n.\] 
    The \emph{base polytope} of $\M$ is defined as
        \[ \mathscr{P}(\M) = \text{convex hull} \{ e_B: B\in \mathscr{B}\} \subseteq \mathbb{R}^n.\]
\end{defi}

If instead of taking the convex hull of the indicator vectors of the bases of the matroid one considers all independent subsets, the resulting polytope $\mathscr{P}_I(\M)$ is customarily called the \emph{independence polytope} of $\M$. The study of matroids from the polyhedral point of view was initiated by Edmonds in \cite{edmonds}, and over the past half century has been a prominent area of research in diverse areas of mathematics such as optimization, discrete geometry or tropical geometry.

\begin{obs}
    Since all the bases of a matroid have cardinality equal to the rank of the matroid, $k = \rk(E)$, we see that the base polytope has all of its vertices lying on the hyperplane $\sum_{i=1}^n x_i = k$ in $\mathbb{R}^n$. Hence, its dimension is at most $n-1$. In \cite{feichtner} Feichtner and Sturmfels proved that
        \[ \dim \mathscr{P}(\M) = n - c(\M),\]
    where $c(\M)$ is the number of connected components of the matroid $\M$. In particular a matroid on $n$ elements is connected if and only if $\dim \mathscr{P}(\M) = n - 1$. Also, since $\mathscr{P}(\M_1\oplus \M_2) = \mathscr{P}(\M_1)\times \mathscr{P}(\M_2)$, and the Ehrhart polynomial of a product of polytopes is just the product of their Ehrhart polynomials, we can restrict ourselves only to connected matroids.
\end{obs}

\begin{obs}\label{rango}
    Notice that if $\M$ is a matroid of rank $k$ and cardinality $n$, then $\M^*$ is a matroid of rank $n-k$ and cardinality $n$. Moreover, the polytopes $\mathscr{P}(\M)$ and $\mathscr{P}(\M^*)$ are obtained one from the other via an involution of the form $x\mapsto (1,\ldots,1) - x$. In particular, if we denote by $\ehr(\M,t)$ the Ehrhart polynomial of $\mathscr{P}(\M)$ and analogously for $\M^*$, it is clear that
        \[ \ehr(\M, t) = \ehr(\M^*, t).\]
    
    Therefore, $\M$ is a counterexample to Conjecture \ref{deloera} if and only if $\M^*$ is a counterexample too. Hence, it suffices to look at matroids satisfying $k\leq n-k$ or, equivalently, $2k\leq n$.
\end{obs}

Observe that base polytopes of matroids are $0/1$-polytopes. This, together with a property about the edges characterizes them, due to a fundamental result by Gel'fand, Goresky, MacPherson and Serganova in \cite{ggms}.

\begin{teo}[\cite{ggms}]\label{ggms}
    Let $\mathscr{P}$ be a polytope in $\mathbb{R}^n$. Then $\mathscr{P}$ is the base polytope of a matroid if and only if the following two conditions hold:
    \begin{itemize}
        \item $\mathscr{P}$ is a $0/1$-polytope.
        \item All the edges of $\mathscr{P}$ are of the form $e_i-e_j$ for $i\neq j$.
    \end{itemize}
\end{teo}

\subsection{Generalized permutohedra and polymatroids}

In \cite{postnikov} Postnikov studied certain polytopes obtained as deformations of an usual permutohedron while preserving the edges directions. Since all the edges of an usual permutohedron are parallel to some vector of the form $e_i-e_j$, and conversely, any polytope having edges of that form can be obtained via deforming an usual permutohedron, it is customary to present Postnikov's polytopes in the following way. 

\begin{defi}
    A \emph{generalized permutohedron} is a polytope $\mathscr{P}\subseteq \mathbb{R}^n$ such that all of its edges are parallel to $e_i-e_j$ for $i\neq j$.
\end{defi}

Observe that Theorem \ref{ggms} can be restated in a simpler way: a matroid polytope is just a generalized permutohedron with vertices with $0/1$ coordinates. Soon, it was realised that generalized permutohedra were essentially equivalent to the class of base polytopes of \emph{polymatroids}, a combinatorial structure introduced by Edmonds and Rota four decades before \cite{edmonds-rota}. Let us review one of their many equivalent definitions.

\begin{defi}
    A \emph{polymatroid} $\mathsf{P}=(E,f)$ consists of a finite set $E$ and a map $f:2^E\to \mathbb{R}$ with the following properties:
    \begin{itemize}
        \item $f(\varnothing) = 0$.
        \item If $A_1\subseteq A_2\subseteq E$, then $f(A_1)\leq f(A_2)$.
        \item If $A_1,A_2\subseteq E$, then
            \[ f(A_1)+f(A_2) \geq f(A_1\cup A_2) + f(A_1\cap A_2).\]
    \end{itemize}
\end{defi}

\begin{obs}
    Observe that if we further require that $f(A) \leq |A|$ and that the values that $f$ assume are integers, what we end up obtaining is in fact a matroid having $f$ as rank function. This is why polymatroids are a generalization of matroids indeed.
\end{obs}

Within the framework of combinatorial optimization there are two important polytopes associated with a polymatroid; the study of these polyhedra was pioneered by Edmonds in \cite{edmonds} as well. It is possible to generalize the notions corresponding to the base polytope of a matroid and the independence polytope of a matroid\footnote{We want to emphasize that we are using the name ``base polytope of a polymatroid'' to avoid confusions, given that in combinatorial optimization it is customary to use the term ``polymatroid'' to speak about what we called the independence polytope of a polymatroid.}.

\begin{defi}
    If $\mathsf{P}=(E,f)$ is a polymatroid on $E=\{1,\ldots,n\}$, then its \emph{base polytope} is defined as
         \[ \mathscr{P}(\mathsf{P}) = \left\{x\in \mathbb{R}^n : \sum_{i=1}^n x_i = f(E), \sum_{i\in S} x_i \leq f(S) \text{ for all } S\subseteq E\right\},\]
    and its \emph{independence polytope} is defined as
    \[ \mathscr{P}_I(\mathsf{P}) = \left\{x\in \mathbb{R}_{\geq 0}^n : \sum_{i\in S} x_i \leq f(S) \text{ for all } S\subseteq E\right\}.\]
\end{defi}

Both objects are indeed polytopes, as they are bounded. This is because the base polytope has each coordinate bounded by
    \[ 0\leq f(E)-f(E\smallsetminus\{i\}) \leq x_i \leq f(\{i\}),\]
and the independence polytope by
    \[ 0 \leq x_i \leq f(\{i\}).\]

Moreover, we have that up to a translation to the closed positive orthant, generalized permutohedra and base polytopes of polymatroids coincide.

\begin{teo}
    A polytope $\mathscr{P}\subseteq \mathbb{R}^n$ is the base polytope of a polymatroid if and only if it is a generalized permutohedron and lies in the closed positive orthant.
\end{teo}

A proof of this result can be found in \cite{derksenfink} by Derksen and Fink, in the context of the study of an even broader class of polyhedra called ``megamatroids''. For alternative proofs we refer to \cite{postnikov-reiner-williams} or \cite{castilloliupolymatroids}, where other characterizations of these polytopes are discussed in full detail.

It will not be necessary to deal with generalized permutohedra or polymatroids in their full generality. All of our construction relies on the geometric counterpart of the combinatorics of matroids that we developed in Section \ref{sec:two}.

\section{Ehrhart theory and minimal matroids}\label{sec:four}

The goal of this section is to review the Ehrhart theory of uniform matroids, give a geometric point of view of the operation of circuit-hyperplane relaxation as was done in \cite{ferroni2}, and link everything with the Ehrhart theory of sparse paving matroids.

\subsection{Uniform matroids}

The base polytope of the uniform matroid $\U_{k,n}$ is known in the literature as the \emph{hypersimplex} $\Delta_{k,n}$. In \cite{katzman} Katzman derived an explicit formula for the Ehrhart polynomial of the hypersimplex.

\begin{teo}[\cite{katzman}]\label{katzman}
    The Ehrhart polynomial of the base polytope of the uniform matroid $\U_{k,n}$ is given by
		\begin{equation}\label{formula} 
			\ehr(\U_{k,n},t) = \sum_{j=0}^{k-1} (-1)^j \binom{n}{j} \binom{(k-j)t+n-1-j}{n-1}.
		\end{equation}
\end{teo}

In \cite{deloera} De Loera, Haws and K\"oppe used this formula to prove that all uniform matroids of rank $2$ do have positive Ehrhart coefficients. In \cite{ferroni1} the author provided an alternative proof of Theorem \ref{katzman} and found an explicit combinatorial formula for each of the Ehrhart coefficients of hypersimplices. To state properly such formula, we recall that the \emph{Eulerian numbers} $A(n,k)$ denote the number of permutations on $n$ elements that have exactly $k$ descents. Also, the \emph{Lah number} $L(n,m)$ is defined as the number of ways of partitioning the set $\{1,\ldots,n\}$ into exactly $m$ linearly ordered blocks. If $\pi$ is a partition of $\{1,\ldots,n\}$ into $m$ linearly ordered blocks, we write that $b\in \pi$ when $b$ is a block. So, for example $(2,3)\in \{(2,3),(1)\}$. We define the \emph{weight of $\pi$} by the following formula:
		\[ w(\pi) := \sum_{b\in\pi} w(b),\]
	where $w(b)$ is the number of elements in $b$ that are smaller (as positive integers) than the first element in $b$.

\begin{defi}
	The \emph{weighted Lah number} $W(\ell,n,m)$ is the number of partitions of weight $\ell$ of $\{1,\ldots,n\}$ into exactly $m$ linearly ordered blocks. 
\end{defi}

Using these and the Eulerian numbers we can express all the Ehrhart coefficients of the base polytope of a uniform matroid as follows.

\begin{teo}[\cite{ferroni1}]\label{formula-ehrhart-wlah}
	For every $1\leq k\leq n - 1$ and $0\leq m\leq n-1$, we have that
	\[ [t^m]\ehr(\U_{k,n},t) = \frac{1}{(n-1)!}\sum_{\ell=0}^{k-1} W(\ell,n,m+1) A(m,k-\ell-1).\]
\end{teo}

This formula will be useful to provide an explicit upper bound for some of the Ehrhart coefficients of $\U_{k,n}$.

\subsection{Minimal matroids}

In \cite{ferroni2} Ferroni studied the base polytopes of the so-called ``minimal matroids''. 

\begin{defi}
    The \emph{minimal matroid} $\T_{k,n}$ is defined as the graphic matroid arising from a cycle of length $k+1$ having one of its edges replaced by $n-k$ parallel copies.
\end{defi}

A basic fact is that $\T_{k,n}$ is a connected matroid of cardinality $n$ and rank $k$. Moreover, it can be proved that it has exactly $k(n-k)+1$ bases. These matroids $\T_{k,n}$ play a key role in our construction. They take their name by a result proved independently by Dinolt and Murty in \cite{dinolt} and \cite{murty} respectively, showing that $\T_{k,n}$ is the unique connected matroid (up to isomorphism) having cardinality $n$ and rank $k$ achieving the minimal possible number of bases. 

\begin{ej}
    In Figure \ref{figura} we can see the graph corresponding to the matroid $\T_{5,8}$. 
    
    \begin{figure}[h]
	    \centering
		\begin{tikzpicture}  
		[scale=2.0,auto=center,every node/.style={circle, fill=black, inner sep=2.3pt}] 
		\tikzstyle{edges} = [thick];
		
		\node (a1) at (-0.5,0.86) {};  
		\node (a2) at (-1,0)  {}; 
		\node (a3) at (-0.5,-0.86)  {};  
		\node (a4) at (0.5,-0.86) {};  
		\node (a5) at (1,0)  {};  
		\node (a6) at (0.5,0.86)  {};    
		
		\draw[edges] (a1) -- (a2); 
		\draw[edges] (a2) -- (a3);  
		\draw[edges] (a3) -- (a4);  
		\draw[edges] (a4) -- (a5);  
		\draw[edges] (a5) -- (a6);  
		\draw[edges] (a6) edge[color=red] (a1);
		\draw[edges] (a6) edge[bend right=20,color=red] (a1);
		\draw[edges] (a6) edge[bend right=-20,color=red] (a1);
		\end{tikzpicture} \caption{$\T_{5,8}$}\label{figura}
	\end{figure}
\end{ej}

One of the main results in \cite{ferroni2} is given by the following geometric interpretation of the notion of circuit-hyperplane relaxation. 

\begin{teo}[\cite{ferroni2}]
    Let $\M$ be a connected matroid of rank $k$ and cardinality $n$ with a circuit-hyperplane $H$ and let $\widetilde{\M}$ be the relaxed matroid. Then the polytope $\widetilde{\mathscr{P}}$ of $\widetilde{\M}$ is obtained by stacking the polytope of the minimal matroid $\T_{k,n}$ on a facet of $\mathscr{P}$. 
\end{teo}

In other words, relaxing a circuit-hyperplane consists of gluing two polytopes, one of which corresponds to the base polytope of a minimal matroid.

\subsection{Ehrhart polynomials of sparse paving matroids} 

Given that the relaxation has a nice counterpart at the level of polytopes, it is reasonable to expect a nice relation between the Ehrhart polynomial of the base polytope of a matroid and that of one of its relaxations. Recall that $\ehr(\M,t)$ denotes the Ehrhart polynomial of the base polytope of the matroid $\M$. 

\begin{teo}[\cite{ferroni2}]\label{ehrhartrelax}
    Let $\M$ be a matroid of rank $k$ and cardinality $n$ with a circuit-hyperplane $H$. Denote by $\widetilde{\M}$ the matroid obtained after relaxing $H$. The following equality holds:
		\[ \ehr(\widetilde{\M},t) = \ehr(\M,t)+\ehr(\T_{k,n},t-1).\]
\end{teo}

We want to emphasize that the second summand on the right is evaluated in $t-1$. In \cite[Theorem 1.6]{ferroni2} the following explicit formula for $\ehr(\T_{k,n}, -)$ is derived.

\begin{prop}[\cite{ferroni2}]\label{formulasa}
    The Ehrhart polynomial of the base polytope of the minimal matroid $\T_{k,n}$ is given by:
		\[\ehr(\T_{k,n},t)= \frac{1}{\binom{n-1}{k-1}} \binom{t+n-k}{n-k} \sum_{j=0}^{k-1}\binom{n-k-1+j}{j}\binom{t+j}{j}.\]
\end{prop}

\begin{obs}
    Notice that this expression shows that the coefficients of $\ehr(\T_{k,n},t-1)$ are positive. We will need the positivity of these coefficients to guarantee the existence of large counterexamples to Conjecture \ref{deloera}.
\end{obs}

As a corollary of Theorem \ref{ehrhartrelax} and of the fact that we have explicit formulas for $\ehr(\T_{k,n}, t-1)$ and $\ehr(\U_{k,n}, t)$, we can deduce explicit formulas for the Ehrhart polynomial of all sparse paving matroids. 

\begin{cor}\label{formusp}
    Let $\M$ be a sparse paving matroid having $n$ elements, rank $k$, and exactly $\lambda$ circuit-hyperplanes. Then:
        \[ \ehr(\M, t) = \ehr(\U_{k,n}, t) - \lambda \ehr(\T_{k,n}, t-1).\]
\end{cor}

\begin{proof}
    It is a direct consequence of Theorem \ref{ehrhartrelax}, since relaxing all the $\lambda$ hyperplanes of $\M$ yields the uniform matroid $\U_{k,n}$.
\end{proof}

\begin{obs}
    Since the coefficients of $\ehr(\T_{k,n},t-1)$ are nonnegative, this shows that the Ehrhart polynomial of a sparse paving matroid of rank $k$ and cardinality $n$ is coefficient-wise bounded by the Ehrhart polynomial of $\U_{k,n}$. This supports the upper bound part of \cite[Conjecture 1.5]{ferroni2}.
\end{obs}

Now that we have a good method to compute Ehrhart polynomials for a presumably enormous family of matroids (see Remark \ref{predominance}), it is reasonable to try to search for a potential counterexample to Conjecture \ref{deloera} within that family. This is what we will do in the next section.

\section{Searching a large number of circuit-hyperplanes}\label{sec:five}

The heuristics of our search will be the following. Since $\ehr(\T_{k,n},t-1)$ has positive coefficients, we see in Corollary \ref{formusp} that the coefficients of $\ehr(\M,t)$ are smaller when $\lambda$ is bigger. We will try to find $n$ and $k$ that admit a $\lambda$ sufficiently big to attain a negative coefficient for $\ehr(\M,t)$.

\subsection{Coding theory} 

Sparse paving matroids have a nice relation with two classes of objects that are very interesting on their own, and for which we have plenty of literature to deduce good bounds.

The \emph{Johnson Graph} $J(n,k)$ is the graph having as vertices all the subsets of cardinality $k$ of the set $\{1,\ldots,n\}$, and edges connecting them when their intersection is a set of cardinality $k-1$. It can be seen that $J(n,k)$ is the $1$-skeleton of the hypersimplex $\Delta_{k,n}$.

The following result provides a dictionary between sparse paving matroids, a particular class of binary codes and stable subsets of the Johnson graph.

\begin{teo}\label{equivalence}
    Let $\mathscr{S}$ be a collection of subsets of $\{1,\ldots,n\}$ such that all of its members have cardinality $k$. Then the following are equivalent:
    \begin{enumerate}[(a)]
        \item $\mathscr{S}$ is the set of circuit-hyperplanes of a sparse paving matroid with $n$ elements and rank $k$.
        \item $\mathscr{S}$ is a stable subset of nodes of the Johnson Graph $J(n,k)$.
        \item The set of all the indicator vectors of the elements of $\mathscr{S}$ is the set of words of a binary code such that all words have length $n$, constant weight $k$ and  minimum Hamming distance at least $4$ (i.e. any two distinct words of the code differ in at least $4$ positions).
    \end{enumerate}
\end{teo}

\begin{proof}
    The proof of the equivalence between (a) and (b) can be found in \cite[Lemma 8]{bansal} by Bansal, Pendavingh and van der Pol. We reproduce it here for the sake of completeness. To prove that (a) $\Rightarrow$ (b), assume that $\M$ is sparse paving and that $H_1$ and $H_2$ are two circuit-hyperplanes such that $|H_1\cap H_2| = k - 1$ (i.e. that they correspond to adjacent vertices of $J(n,k)$). Since $\M$ is paving we obtain that $\rk(H_1\cap H_2) = |H_1\cap H_2| = k - 1$. Also, as $H_1$ and $H_2$ are hyperplanes, we have:
        \[ \rk(H_1) + \rk(H_2) = 2(k-1)=2k-2.\]
    But, on the other hand, as $H_1\cup H_2$ must have rank $k$, we have: 
        \[ \rk(H_1\cap H_2) + \rk(H_1\cup H_2) = (k-1) + k = 2k-1.\]
    All this information together implies that
        \[ \rk(H_1)+\rk(H_2) < \rk(H_1\cup H_2) + \rk(H_1\cap H_2),\]
    which cannot happen for a matroid. Now, to prove that (b) implies (a), assume that $\mathscr{S}$ is a stable subset of the Johnson Graph $J(n,k)$, and consider the collection $\mathscr{B}$ of all subsets of $\{1,\ldots,n\}$ of cardinality $k$ that are not in $\mathscr{S}$. We trivially have that $\mathscr{B}\neq\varnothing$ because the Johnson graph contains edges. Assume that the basis-exchange-property does not hold. There exist two sets $B_1$ and $B_2$ in $\mathscr{B}$ and an element $x\in B_1\smallsetminus B_2$ such that $(B_1\smallsetminus\{x\})\cup\{y\}\notin \mathscr{B}$ for all $y\in B_2\smallsetminus B_1$. Observe that this implies that $|B_2\smallsetminus B_1| > 1$, because otherwise it would be $(B_1\smallsetminus\{x\})\cup\{y\}=B_2\in\mathscr{B}$ for the only $y\in B_2\smallsetminus B_1$. Now, let us choose two distinct elements $y,z\in B_2\smallsetminus B_1$ and consider the sets $H_1 = (B_1\smallsetminus\{x\})\cup\{y\}$ and $H_2 = (B_1\smallsetminus\{x\})\cup\{z\}$. Since $H_1$ and $H_2$ are not in $\mathscr{B}$, it holds that they are in $\mathscr{S}$. This is a contradiction, because $|H_1\cap H_2| = |B_1\smallsetminus\{x\}|=k-1$ which contradicts that $\mathscr{S}$ was stable.
    
    The equivalence between (b) and (c) follows from the definitions: an edge on the Johnson graph corresponds to two words that have Hamming distance equal to $2$.
\end{proof}

In \cite[Theorem 26]{schroter} Joswig and Schr\"oter stated an extended version of these equivalences in the context of the study of a class of matroids called ``split matroids''.

Now we will use the fact that there are binary codes of length $n$, constant weight $k$ and minimum Hamming distance at least $4$ that contain \emph{many} words. Although there are several constructive proofs for such codes with an even larger number of words for special cases of $n$ or $k$, the bounds we will use here suffice for our purposes. The statement and the proof of the following result are due to Graham and Sloane and can be found in \cite{grahamsloane}.

\begin{teo}[\cite{grahamsloane}]\label{codes}
    There exists a binary code $\mathscr{S}$ with words of length $n$, constant weight $k$, Hamming distance at least $4$, and such that $|\mathscr{S}| \geq \frac{1}{n}\binom{n}{k}$.
\end{teo}

\begin{proof}
    Let us denote by $\mathbb{F}^n_{2,k}$ the set of all binary words of length $n$ and constant weight $k$, and by $\mathbb{Z}_n$ the set of integers modulo $n$. Consider the map:
        \[ T: \mathbb{F}^n_{2,k} \to \mathbb{Z}_n,\]
        \[ T(a_1, \ldots, a_n) = \sum_{i=1}^n (i-1)a_i \pmod{n}.\]
    For each $i=0,\ldots, n-1$ let us call $C_i = T^{-1}(\{i\})$. We claim that for each $C_i$, the minimum distance between two of its words is at least $4$. Assume on the contrary that there are two distinct words $\mathbf{a} = (a_1,\ldots,a_n)$ and $\mathbf{b}=(b_1,\ldots, b_n)$ at distance less than $4$ in $C_i$. Since both words have the same weight, we have that their distance is exactly $2$. Also, we see that there must exist two positions, say $r$ and $s$, such that $a_r = 1$ and $b_r = 0$ and $a_s=0$ and $b_s=1$. But observe that
        \begin{align*} 
            T(\mathbf{a}) &= x + r = i  \pmod{n},\\
            T(\mathbf{b}) &= x + s = i \pmod{n},
        \end{align*}
    for a certain $x\in \mathbb{Z}_n$. This implies that $r\equiv s\pmod{n}$ which is clearly impossible. Thus, the minimum distance between words of $C_i$ is $4$. Now, since:
        \[ \binom{n}{k} = |\mathbb{F}^n_{2,k}| = \sum_{i=0}^{n-1} |C_i|,\]
    we see that there has to be at least one $i$ such that $|C_i|\geq \frac{1}{n} \binom{n}{k}$.
\end{proof}

\subsection{An explicit counterexample}

Let us show how our set-up allows us to construct a counterexample to Conjectures \ref{deloera} and \ref{castilloliu}.

\begin{teo}
    There exists a sparse paving matroid $\M$ with $20$ elements, rank $9$ and having $8398$ circuit-hyperplanes, and hence having Ehrhart polynomial with negative quadratic and cubic coefficients.
\end{teo}

\begin{proof}
    By Theorem \ref{codes} there exists a binary code of length $20$, constant weight $9$ and Hamming distance at least $4$, having at least $\frac{1}{20}\binom{20}{9} = 8398$ words. In fact, it can be proved that for the particular choice of $n=20$ and $k=9$ all the $C_i$'s in the proof of Theorem \ref{codes} have cardinality $8398$. In particular, by choosing for instance $C_0$ as our code, we have a code with $8398$ words.
    
    By the equivalence between (a) and (c) in Theorem \ref{equivalence}, we get that there is a sparse paving matroid $\M$ with $20$ elements and rank $9$ that has exactly $8398$ circuit-hyperplanes. 
    
    Now, using the formula of Corollary \ref{formusp}, we obtain that
    \[\ehr(\M, t) = \ehr(\U_{9,20}, t) - 8398 \ehr(\T_{9,20}, t-1),\]
    and we can compute this polynomial explicitly and see that its quadratic coefficient is $-\frac{142179543511}{15437822400} < 0$ and its cubic coefficient is $-\frac{4816883312963}{51459408000} < 0$.
\end{proof}

\section{Constructing counterexamples with small rank}\label{sec:six}

The goal of this section is to extend the construction of our counterexample in order to prove Theorems \ref{greaterthan9} and \ref{main3}.

Experimentation with several values of $n$ and $k$ shows that the most well-behaved coefficient for our purposes is the quadratic one. In other words, we experimentally observed that in the vast majority of cases, when a matroid is not Ehrhart positive, in particular its quadratic Ehrhart coefficient is negative. See Remark \ref{cubic} for an example of a matroid whose Ehrhart polynomial has a negative cubic coefficient, and having the property that the remaining coefficients are positive.

The idea is to give a good lower bound for $[t^2]\ehr(\T_{k,n},t-1)$ and a good upper bound for $[t^2]\ehr(\U_{k,n},t)$ that allow us to work more comfortably. 

\subsection{A lower bound for minimal matroids} We start with a precise expression for the quadratic coefficient of $\ehr(\T_{k,n},t-1)$.

\begin{lema}\label{formutkn2}
    For every $1\leq k\leq n - 1$, the quadratic coefficient of $\ehr(\T_{k,n},t-1)$ is given by:
    \[ [t^2]\ehr(\T_{k,n},t-1) = \frac{1}{\binom{n-1}{k-1}} \left(  \stirling{n-k}{2} \frac{1}{(n-k)!} + \frac{1}{n-k} \sum_{j=1}^{k-1} \frac{1}{j} \binom{n-k-1+j}{j}\right),\] 
    where the brackets denote the unsigned Stirling numbers of the first kind.
\end{lema}

\begin{proof}
    Assume that $a$ is a nonnegative integer and consider the polynomial in $t$ given by $\binom{t+a-1}{a}$. For $a\geq 1$ it is a multiple of $t$ with linear term equal to $\frac{(a-1)!}{a!}=\frac{1}{a}$. On the other hand, for $a=0$, it is constantly $1$. Using Proposition \ref{formulasa} we see that 
    \[
        \binom{n-1}{k-1}\ehr(\T_{k,n},t-1) = \binom{t+n-k-1}{n-k} \sum_{j=0}^{k-1}\binom{n-k-1+j}{j}\binom{t+j-1}{j}.\]
    Since $n-k\geq 1$, the first factor on the right is a multiple of $t$. However, for $j=0$ the first term of the sum is identically $1$. Hence, the total quadratic factor of the above polynomial can be recovered as    
    \[[t^2]\binom{t+n-k-1}{n-k}+ \sum_{j=1}^{k-1} \binom{n-k-1+j}{j}[t^1]\binom{t+j-1}{j}\cdot[t^1]\binom{t+n-k-1}{n-k},\]
    which, after dividing by $\binom{n-1}{k-1}$, reduces to the expression of the statement, because $[t^2]\binom{t+n-k-1}{n-k} = \frac{1}{(n-k)!}\tstirling{n-k}{2}$ and, as we said before, $[t^1]\binom{t+a-1}{a} = \frac{1}{a}$.
\end{proof}

Using the preceding Lemma we can give a nice lower bound for the quadratic coefficient of $\ehr(\T_{k,n},t-1)$, essentially by just ignoring many of the terms appearing in the expression we just obtained.

\begin{prop}\label{cotainf}
    The quadratic coefficient of $\ehr(\T_{k,n},t-1)$ satisfies:
    \[ [t^2]\ehr(\T_{k,n},t-1) \geq \frac{1}{k(n-1)}.\]
\end{prop}

\begin{proof}
    Observe that in the sum inside the parentheses in the formula of Lemma \ref{formutkn2}, we can pick only the term corresponding to $j=k-1$ and forget the rest. Hence:
    \begin{align*}
        [t^2]\ehr(\T_{k,n},t-1) &\geq \frac{1}{\binom{n-1}{k-1}} \left( \stirling{n-k}{2} \frac{1}{(n-k)!} + \frac{1}{(n-k)(k-1)} \binom{n-2}{k-1}\right)\\
        &\geq \frac{1}{\binom{n-1}{k-1}}\cdot  \frac{1}{(n-k)(k-1)} \binom{n-2}{k-1}\\
        &= \frac{1}{(k-1)(n-1)}\\
        &\geq \frac{1}{k(n-1)},
    \end{align*}
    where in the second to last step we just expanded the binomial coefficients and canceled many factors.
\end{proof}

\subsection{An upper bound for uniform matroids}

Before establishing an upper bound for the quadratic coefficient of $\ehr(\U_{k,n},t)$, we state some formulas that relate the Stirling numbers of the first kind with the \emph{harmonic} numbers. If we denote by $H^{(k)}_n$ the number $1+\frac{1}{2^k}+\ldots + \frac{1}{n^k}$, and $H_n = H^{(1)}_n$, we have the following identities: 
\begin{align}
    \frac{1}{(n-1)!} \stirling{n}{1} &= 1,\\
    \frac{1}{(n-1)!} \stirling{n}{2} &= H_{n-1},\\
    \frac{1}{(n-1)!} \stirling{n}{3} &= \frac{1}{2} \left( H_{n-1}^2 - H_{n-1}^{(2)}\right).
\end{align}
None of these identities is difficult to prove. For further reading on the relation between the harmonic numbers and the Stirling numbers of the first kind, we refer to Graham, Knuth and Patashnik's book \cite{grahamknuth}.

\begin{prop}\label{cotasup}
    The quadratic coefficient of $\ehr(\U_{k,n},t)$ satisfies:
        \[ [t^2] \ehr(\U_{k,n},t) \leq \frac{\binom{k+1}{2}+\binom{k}{2}}{(n-1)!}\cdot\stirling{n}{3} \leq \binom{k+1}{2} H_{n-1}^2.\]
\end{prop}

\begin{proof}
    We will use Theorem \ref{formula-ehrhart-wlah}. For the quadratic term of the Ehrhart polynomial of a uniform matroid it holds:
    \begin{equation}\label{wlah}
        [t^2]\ehr(\U_{k,n},t) = \frac{1}{(n-1)!} \left( W(k-1,n,3) + W(k-2,n,3)\right).
    \end{equation} 
    
    Recall that the number $W(\ell, n, 3)$ denotes the number of ways of partitioning the set $\{1,\ldots,n\}$ into $3$ linearly ordered blocks having total weight $\ell$. 
    
    Let us prove that 
    \begin{equation}\label{deswlah}W(\ell, n, 3) \leq \binom{\ell+2}{2} W(0,n,3).\end{equation} To this end, let us start with a partition of $\{1,\ldots,n\}$ into $3$ blocks having total weight $0$. Consider the operation consisting of the following three steps:
    \begin{itemize}
        \item Swap the elements in the first position of the first block with the $x$-th smallest element of the first block.
        \item Swap the elements in the first position of the second block with the $y$-th smallest element of the second block.
        \item  Swap the elements in the first position of the third block with the $z$-th smallest element of the third block.
    \end{itemize}
    If $(x-1)+(y-1)+(z-1) = \ell$ what we obtain is a partition of $\{1,\ldots,n\}$ into three blocks having total weight $\ell$. Observe that we can do this in at most $\binom{\ell+2}{2}$ ways (the number of ways of putting $\ell$ balls into $3$ boxes). Also, in this way we can achieve all the possible partitions of weight $\ell$. It is clear how to deduce the inequality \eqref{deswlah} from this fact. Now, we know that $W(0,n,m) = \tstirling{n}{m}$ (see \cite[Remark 3.8]{ferroni1}). If we use the formula of equation \eqref{wlah}, we get the first inequality in our statement. Also, since 
    \[\frac{1}{(n-1)!} \stirling{n}{3} = \frac{1}{2} \left( H_{n-1}^2 - H_{n-1}^{(2)}\right)\leq \frac{1}{2}H_{n-1}^2,\]
    it is easy to conclude the second inequality of the statement (we also used that $\binom{k}{2} \leq \binom{k+1}{2}$ to get a simpler form of the right-hand-side).
\end{proof}

\subsection{The main theorems} 

We can use our bounds to show the existence of counterexamples on every rank $k\geq 3$. The following is a refined form of Theorem \ref{main3}; its proof also settles most part of Theorem \ref{greaterthan9}.

\begin{teo}\label{thm:el-clave}
    If $n\geq 3589$ and $3\leq k \leq n - 3$ then there exists a matroid of rank $k$ and cardinality $n$ that is not Ehrhart positive. For $4\leq k \leq n-4$ we may choose $n\geq 104$.
\end{teo}

\begin{proof}
    From Theorem \ref{codes} and the equivalence between (a) and (c) in Theorem \ref{equivalence}, we have that there exists a sparse paving matroid $\M$ of rank $k$ and cardinality $n$, having $\lambda\geq \frac{1}{n} \binom{n}{k}$
    circuit-hyperplanes. We know by Corollary \ref{formusp} that
        \[ \ehr(\M,t) = \ehr(\U_{k,n},t) - \lambda \ehr(\T_{k,n},t-1),\]
    and thus
    \begin{align*}
        [t^2]\ehr(\M,t) &= [t^2]\ehr(\U_{k,n},t) - \lambda [t^2]\ehr(\T_{k,n},t-1)\\
        &\leq [t^2]\ehr(\U_{k,n},t) - \frac{1}{n}\binom{n}{k} [t^2]\ehr(\T_{k,n},t-1)\\
        &\leq \binom{k+1}{2}H_{n-1}^2 - \frac{1}{n}\binom{n}{k} \frac{1}{k(n-1)},
    \end{align*}
    where we used Lemmas \ref{cotainf} and \ref{cotasup}. It suffices to analyze when the following inequality is achieved:
        \begin{equation}\label{harmonic} \binom{k+1}{2} H_{n-1}^2 < \frac{1}{n}\binom{n}{k}\frac{1}{k(n-1)}.\end{equation}
    Let us split into some cases:
    \begin{itemize}
        \item If $k=3$, \eqref{harmonic} becomes
        \[ 6H_{n-1}^2 \leq \frac{1}{3n(n-1)} \binom{n}{3}.\]
        Since $H_{n-1}^2 \sim \log(n)^2$, we see that the right-hand-side grows much faster than the left-hand-side. In particular, the inequality holds for all $n\geq 10439$. Also, we can verify with a computer the following finite cases $3589 \leq n \leq 10438$ and see that for all of them one has $[t^2] \ehr(\M,t) < 0$. This proves that there exist counterexamples of rank $3$ for all $n\geq 3589$.
        \item If $k=4,5,6,7,8$, analogous considerations show that for $n\geq 104$, we can always find such counterexamples.
        \item If $k\geq 9$, recalling that in Remark \ref{rango} we stated that the Ehrhart polynomial of a matroid is equal to that of its dual, we can assume that $2k\leq n$ and consider a stronger version of inequality \eqref{harmonic}:
            \[ \binom{n+1}{2} H_{n-1}^2 < \frac{1}{n} \binom{n}{9} \frac{1}{n(n-1)},\]
        which holds for all $n\geq 55$. By checking manually the cases $n=20,\ldots, 54$, we prove also a major part of Theorem \ref{greaterthan9} (the case $n=19$ is addressed in the last section).\qedhere
    \end{itemize}
\end{proof}

\section{The rank two case}\label{sec:seven}

In a previous version of this manuscript it was proved that the Ehrhart polynomials of the base polytopes of all sparse paving matroids of rank $2$ have positive coefficients. After it appeared, Ferroni, Jochemko and Schr\"oter proved in \cite{fjs} that in fact \emph{all} matroids of rank $2$ are Ehrhart positive using a much shorter and insightful approach.

We will sketch here how the original proof of the Ehrhart positivity of sparse paving of rank $2$ matroids was achieved in the first place. The proof of a technical inequality involving Stirling numbers (which can be avoided using the approach of \cite{fjs}) is omitted.

\subsection{Stable sets in the Johnson graph} 

It is worth noticing that the maximum stable set that one can construct on the Johnson Graph $J(n,2)$ has cardinality $\lfloor \frac{n}{2}\rfloor$. In fact, two nodes are not connected if and only if they correspond to disjoint $2$-sets. If we choose any $\lfloor \frac{n}{2}\rfloor + 1$ nodes, they correspond to that number of $2$-subsets of $\{1,\ldots,n\}$ and, among all of their $2\lfloor\frac{n}{2}\rfloor + 2 > n$ elements, there must be at least one repetition. This says that such a set is not stable. By choosing the sets $\{1,2\},\{3,4\},\ldots$ we can construct a stable set of cardinality $\lfloor\frac{n}{2}\rfloor$. 

In particular, if $\M$ is a sparse paving matroid of cardinality $n$ and rank $2$, the maximum number of circuit-hyperplanes $\M$ can have is $\lfloor\frac{n}{2}\rfloor$.

\subsection{Ehrhart positivity for rank two sparse paving matroids}

The result that gives us the desired Ehrhart positivity is the following.

\begin{prop}\label{posi}
    For every integer $n\geq 3$ the polynomial
    \[P_n(t) = \ehr(\U_{2,n},t) - \left\lfloor\frac{n}{2}\right\rfloor \ehr(\T_{2,n},t-1)\]
    has positive coefficients.
\end{prop}

\begin{proof}[Sketch of proof]
   By Theorem \ref{formula-ehrhart-wlah} we have that
        \begin{align*}
            [t^m]\ehr(\U_{2,n}, t) &= \frac{1}{(n-1)!} \left(W(0,n,m+1)A(m,1) + W(1,n,m+1)A(m,0)\right)\\
            &= \frac{1}{(n-1)!} \left( W(0,n,m+1) (2^m-m-1) + W(1,n,m+1)\right).
        \end{align*}
    From \cite[Proposition 3.10]{ferroni1} we see that $W(1,n,m+1) \geq (n-1)W(0,n-1,m+1)$. Hence:
    \[ [t^m]\ehr(\U_{2,n},t) \geq \frac{1}{(n-1)!}\left( \stirling{n}{m + 1} (2^m-m-1) + (n-1) \stirling{n - 1} {m + 1}\right).\]
    Working with the formula of Proposition \ref{formulasa} when $k=2$ we can get:
    \[ [t^m] \ehr(\T_{2,n},t-1) = \frac{1}{(n-1)!} \left( \stirling{n-2}{m} + (n-2)\stirling{n-2}{m-1} \right).\]
    Hence, to prove that the coefficients of $P_n(t)$ are positive, it suffices to show that for each $0\leq m\leq n-1$, the following inequality holds:
        \begin{equation}\label{target}
        \stirling{n}{m + 1} (2^m-m-1) + (n-1) \stirling{n - 1}{m + 1} \geq \frac{n}{2} \left( \stirling{n-2}{m} + (n-2)\stirling{n-2}{m-1} \right).\end{equation}
    The presence of the factor $2^m-m-1$ on the left-hand-side multiplying a Stirling number that is bigger than both of the ones appearing on the right in \eqref{target} helps to build an intuition: the expression on the left is likely to be much larger than the expression on the right. However, to give a full and rigorous proof of this inequality, several manipulations are required using the recurrence that Stirling numbers of the first kind satisfy, along with the fact that they form a unimodal and log-concave sequence when either fixing the upper or the lower parameter; see \cite{sibuya} for a thorough list of inequalities that can be used to prove \eqref{target}.
\end{proof}

\begin{cor}
    If $\M$ is a sparse paving matroid of rank $2$, then $\M$ is Ehrhart positive.
\end{cor}

\section{Final remarks}\label{sec:eight}

We have proved that for every $k\geq 3$ there is a (connected) matroid of rank $k$ that is not Ehrhart positive. We have also outlined a proof of the fact that we cannot construct a counterexample within the family of sparse paving matroids on rank $2$.

Also, from the proof of Theorem \ref{thm:el-clave} it follows that for all $n\geq 20$ there is a (connected) matroid of cardinality $n$ that is not Ehrhart positive. It is natural to ask if there are smaller counterexamples.

In fact, for small values of $k$ and $n$ there exist much better bounds and many precise values for the maximum number of circuit-hyperplanes that a sparse paving matroid of rank $k$ and cardinality $n$ can have. See for instance \cite[Table 2]{brouwer} by Brouwer and Etzion. Using these values, one can prove that there exists a sparse paving matroid with $19$ elements, rank $9$ and having $6726$ circuit-hyperplanes that is not Ehrhart positive.

We can rule out the existence of sparse paving matroids with less than $18$ elements. To prove this, it suffices to give a good enough upper bound for the maximum number of circuit-hyperplanes a matroid of rank $k$ and $n$ elements can have.

\begin{lema}\label{numbercircuithyperplanesparsepaving}
    Let $\M$ be a sparse paving matroid of rank $k$ having $n$ elements. Then, the number of circuit-hyperplanes $\lambda$ of $\M$ satisfies the following inequality:
        \begin{equation}\label{lambdakn} \lambda \leq \binom{n}{k}\min \left\{\frac{1}{k+1}, \frac{1}{n-k+1}\right\}.\end{equation}
\end{lema}

\begin{proof}
    If $\M$ is paving, in particular all the $\binom{n}{k-1}$ subsets of cardinality $k-1$ are independent. Let us form a bipartite graph where one of the parts has a node for each independent set of cardinality $k-1$ and the other part has the bases of the matroid $\M$, where we put an edge connecting an independent set $I$ with a basis $B$ whenever $I\subseteq B$. Since an independent set $I$ of rank $k-1$ is contained in a unique hyperplane (the flat spanned by $I$ itself), it follows that either $I$ is a hyperplane or $I\subsetneq H$ for a unique $H$ hyperplane. In the latter case, $|H|\geq k$ and since $\M^*$ is paving, this implies that $|H| = k$, so that $H$ is a circuit-hyperplane. Summarizing, each of the nodes of our graph corresponding to independent sets of cardinality $k-1$ has degree $n-k+1$ (when $I$ is itself a hyperplane) or $n-k$ (when $I$ is contained in a unique circuit-hyperplane). In particular, the number of edges of the whole graph is at least $(n-k) \binom{n}{k-1}$. However, by looking at the nodes corresponding to the bases, we know that each basis has degree $k$, so that the number of edges is exactly $k|\mathscr{B}(\M)|$. Hence:
        \[ (n-k)\binom{n}{k-1} \leq k |\mathscr{B}(\M)|,\]
    which translates into
        \[ \left( 1 - \frac{1}{n-k+1}\right) \binom{n}{k} \leq |\mathscr{B}(\M)|.\]
    Since the number of circuit hyperplanes is $\lambda = \binom{n}{k} - |\mathscr{B}(\M)|$, it follows that
        \[ \lambda \leq \frac{1}{n-k+1} \binom{n}{k}.\]
    Finally, using the same reasoning that we used above but with $\M^*$ instead of $\M$, as the number of circuit-hyperplanes is the same for $\M$ and $\M^*$ (see Remark \ref{circuitcohyperplane}), it follows also that
        \[ \lambda \leq \frac{1}{k + 1} \binom{n}{n-k}= \frac{1}{k+1}\binom{n}{k},\]
    from where one concludes the inequality of the statement.
\end{proof}

\begin{cor}
    If $\M$ is a sparse paving matroid on $n\leq 17$ elements, then $\M$ is Ehrhart positive.
\end{cor}

\begin{proof}
    Let us denote by $\lambda_{k,n}$ the expression on the right-hand-side of inequality \eqref{lambdakn}. Calculating explicitly the polynomials $\ehr(\U_{k,n}, t) - \lambda_{k,n} \ehr(\T_{k,n},t-1)$ for $1\leq k \leq n \leq 17$, we can see that they all have positive coefficients.
\end{proof}

\begin{obs}
    According to \cite{brouwer} the maximum size that a stable set in the Johnson Graph $J(18,9)$ can have is at least $3540$, which improves the bound coming from Theorem \ref{codes}, $\frac{1}{18}\binom{18}{9}=2702$. However, using the bound from Lemma \ref{numbercircuithyperplanesparsepaving} we get that this quantity is at most $4862$. A sharper inequality using the so-called ``Johnson bound'' yields that, in fact, this quantity is less or equal than $4420$. In other words, we know that the maximum number of circuit-hyperplanes that a matroid on $18$ elements and rank $9$ can have lies between $3540$ and $4420$, and this seems to be the best we can currently assert for $k=9$ and $n=18$ (see \cite{upper}). However,
        \[ \ehr(\U_{9,18},t) - 4240 \ehr(\T_{9,18},t-1)\]
    has a negative cubic coefficient. This implies that if we could improve our $3540$ to a $4240$, then there would be a matroid on $18$ elements that is not Ehrhart positive.
\end{obs}

\begin{obs}\label{cubic}
    It is not true that if a matroid has a negative Ehrhart coefficient, then in particular the quadratic coefficient must be negative. For example our construction yields a matroid with $22$ elements, rank $7$ and $7752$ circuit-hyperplanes that has a negative coefficient only on degree $3$. 
\end{obs}

\begin{obs}
    In \cite{ferroni2} the author conjectured that the Ehrhart $h^*$-polynomial of the base polytope of a matroid is always real-rooted. This is a stronger form of another conjecture by De Loera, Haws and K\"oppe in \cite{deloera} asserting the unimodality of the coefficients of the corresponding $h^*$-vector. By using the upper bounds for the number of circuit-hyperplanes of a sparse paving matroid of Lemma \ref{numbercircuithyperplanesparsepaving}, we have verified using a computer that up to $35$ elements all sparse paving matroids are $h^*$-real-rooted.
\end{obs}

\section{Acknowledgements}

The author wants to thank the reviewers of the article for many important suggestions and comments they made to improve several parts of the exposition. Also, he wants to express his gratitude to Matthias Beck, Federico Castillo, Katharina Jochemko, Victor Reiner, Benjamin Schr\"oter and Mario Sanchez for the extremely useful comments and remarks they made on different versions of this manuscript and which served to enhance several aspects of it.

\bibliographystyle{amsalpha}
\bibliography{bibliography}

\providecommand{\bysame}{\leavevmode\hbox to3em{\hrulefill}\thinspace}
\providecommand{\MR}{\relax\ifhmode\unskip\space\fi MR }
\providecommand{\MRhref}[2]{%
  \href{http://www.ams.org/mathscinet-getitem?mr=#1}{#2}
}
\providecommand{\href}[2]{#2}
\begin{thebibliography}{MNWW11}

\bibitem[ABD10]{ardilabenedetti}
Federico Ardila, Carolina Benedetti, and Jeffrey Doker, \emph{Matroid polytopes
  and their volumes}, Discrete Comput. Geom. \textbf{43} (2010), no.~4,
  841--854. \MR{2610473}

\bibitem[AVZ00]{upper}
Erik Agrell, Alexander Vardy, and Kenneth Zeger, \emph{Upper bounds for
  constant-weight codes}, IEEE Trans. Inform. Theory \textbf{46} (2000), no.~7,
  2373--2395. \MR{1806807}

\bibitem[BE11]{brouwer}
Andries~E. Brouwer and Tuvi Etzion, \emph{Some new distance-4 constant weight
  codes}, Adv. Math. Commun. \textbf{5} (2011), no.~3, 417--424. \MR{2831612}

\bibitem[BPvdP15]{bansal}
Nikhil Bansal, Rudi~A. Pendavingh, and Jorn~G. van~der Pol, \emph{On the number
  of matroids}, Combinatorica \textbf{35} (2015), no.~3, 253--277. \MR{3367125}

\bibitem[BR15]{beck}
Matthias Beck and Sinai Robins, \emph{Computing the continuous discretely},
  second ed., Undergraduate Texts in Mathematics, Springer, New York, 2015,
  Integer-point enumeration in polyhedra, With illustrations by David Austin.
  \MR{3410115}

\bibitem[BS18]{becksanyal}
Matthias Beck and Raman Sanyal, \emph{Combinatorial reciprocity theorems},
  Graduate Studies in Mathematics, vol. 195, American Mathematical Society,
  Providence, RI, 2018, An invitation to enumerative geometric combinatorics.
  \MR{3839322}

\bibitem[CL18]{castilloliu}
Federico Castillo and Fu~Liu, \emph{Berline-{V}ergne valuation and generalized
  permutohedra}, Discrete Comput. Geom. \textbf{60} (2018), no.~4, 885--908.
  \MR{3869454}

\bibitem[CL20]{castilloliupolymatroids}
\bysame, \emph{{Deformation Cones of Nested Braid Fans}}, {Int. Math. Res.
  Not.} (2020).

\bibitem[CL21]{castillo2020todd}
\bysame, \emph{On the {T}odd class of the permutohedral variety}, Algebr. Comb.
  \textbf{4} (2021), no.~3, 387--407. \MR{4275820}

\bibitem[DF10]{derksenfink}
Harm Derksen and Alex Fink, \emph{Valuative invariants for polymatroids}, Adv.
  Math. \textbf{225} (2010), no.~4, 1840--1892. \MR{2680193}

\bibitem[Din71]{dinolt}
George~W. Dinolt, \emph{An extremal problem for non-separable matroids},
  Th\'{e}orie des matro\"{\i}des ({R}encontre {F}ranco-{B}ritannique, {B}rest,
  1970), 1971, pp.~31--49. Lecture Notes in Math. Vol. 211. \MR{0389626}

\bibitem[DLHK09]{deloera}
Jes\'{u}s~A. De~Loera, David~C. Haws, and Matthias K\"{o}ppe, \emph{Ehrhart
  polynomials of matroid polytopes and polymatroids}, Discrete Comput. Geom.
  \textbf{42} (2009), no.~4, 670--702. \MR{2556462}

\bibitem[Edm70]{edmonds}
Jack Edmonds, \emph{Submodular functions, matroids, and certain polyhedra},
  Combinatorial {S}tructures and their {A}pplications ({P}roc. {C}algary
  {I}nternat. {C}onf., {C}algary, {A}lta., 1969), Gordon and Breach, New York,
  1970, pp.~69--87. \MR{0270945}

\bibitem[Ehr62]{ehrhart}
Eug\`ene Ehrhart, \emph{Sur les poly\`edres rationnels homoth\'{e}tiques \`a
  {$n$} dimensions}, C. R. Acad. Sci. Paris \textbf{254} (1962), 616--618.
  \MR{130860}

\bibitem[ER66]{edmonds-rota}
Jack Edmonds and Gian-Carlo Rota, \emph{Submodular set functions}, Waterloo
  combinatorics conference, University of Waterloo, Waterloo, Ontario, 1966.

\bibitem[Fer21a]{ferroni1}
Luis Ferroni, \emph{Hypersimplices are {E}hrhart positive}, J. Combin. Theory
  Ser. A \textbf{178} (2021), Paper No. 105365, 13. \MR{4179055}

\bibitem[Fer21b]{ferroni3}
\bysame, \emph{Integer point enumeration on independence polytopes and
  half-open hypersimplices}, Discrete Math. \textbf{344} (2021), no.~8, Paper
  No. 112446, 6. \MR{4252106}

\bibitem[Fer21c]{ferroni2}
\bysame, \emph{On the {E}hrhart polynomial of minimal matroids}, {Discrete
  Comput. Geom.} (2021).

\bibitem[FJS21]{fjs}
Luis {Ferroni}, Katharina {Jochemko}, and Benjamin {Schr{\"o}ter},
  \emph{{Ehrhart polynomials of rank two matroids}}, arXiv e-prints (2021),
  arXiv:2106.08183.

\bibitem[FS05]{feichtner}
Eva~Maria Feichtner and Bernd Sturmfels, \emph{Matroid polytopes, nested sets
  and {B}ergman fans}, Port. Math. (N.S.) \textbf{62} (2005), no.~4, 437--468.
  \MR{2191630}

\bibitem[GGMS87]{ggms}
I.~M. Gel'fand, R.~M. Goresky, R.~D. MacPherson, and V.~V. Serganova,
  \emph{Combinatorial geometries, convex polyhedra, and {S}chubert cells}, Adv.
  in Math. \textbf{63} (1987), no.~3, 301--316. \MR{877789}

\bibitem[GKP94]{grahamknuth}
Ronald~L. Graham, Donald~E. Knuth, and Oren Patashnik, \emph{Concrete
  mathematics}, second ed., Addison-Wesley Publishing Company, Reading, MA,
  1994, A foundation for computer science. \MR{1397498}

\bibitem[GS80]{grahamsloane}
Ronald~L. Graham and Neil J.~A. Sloane, \emph{Lower bounds for constant weight
  codes}, IEEE Trans. Inform. Theory \textbf{26} (1980), no.~1, 37--43.
  \MR{560390}

\bibitem[HHTY19]{hibi-negative}
Takayuki Hibi, Akihiro Higashitani, Akiyoshi Tsuchiya, and Koutarou Yoshida,
  \emph{Ehrhart polynomials with negative coefficients}, Graphs Combin.
  \textbf{35} (2019), no.~1, 363--371. \MR{3898396}

\bibitem[JR21]{jochemko}
Katharina Jochemko and Mohan Ravichandran, \emph{Generalized permutahedra:
  {M}inkowski linear functionals and {E}hrhart positivity}, 2021.

\bibitem[JS17]{schroter}
Michael Joswig and Benjamin Schr\"{o}ter, \emph{Matroids from hypersimplex
  splits}, J. Combin. Theory Ser. A \textbf{151} (2017), 254--284. \MR{3663497}

\bibitem[Kat05]{katzman}
Mordechai Katzman, \emph{The {H}ilbert series of algebras of the {V}eronese
  type}, Comm. Algebra \textbf{33} (2005), no.~4, 1141--1146. \MR{2136691}

\bibitem[Liu19]{liu}
Fu~Liu, \emph{On positivity of {E}hrhart polynomials}, Recent trends in
  algebraic combinatorics, Assoc. Women Math. Ser., vol.~16, Springer, Cham,
  2019, pp.~189--237. \MR{3969575}

\bibitem[LT19]{liuorder}
Fu~Liu and Akiyoshi Tsuchiya, \emph{Stanley's non-{E}hrhart-positive order
  polytopes}, Adv. in Appl. Math. \textbf{108} (2019), 1--10. \MR{3933317}

\bibitem[McM77]{mcmullen}
Peter McMullen, \emph{Valuations and {E}uler-type relations on certain classes
  of convex polytopes}, Proc. London Math. Soc. (3) \textbf{35} (1977), no.~1,
  113--135. \MR{448239}

\bibitem[MNWW11]{mayhew}
Dillon Mayhew, Mike Newman, Dominic Welsh, and Geoff Whittle, \emph{On the
  asymptotic proportion of connected matroids}, European J. Combin. \textbf{32}
  (2011), no.~6, 882--890. \MR{2821559}

\bibitem[Mur71]{murty}
U.~S.~R. Murty, \emph{On the number of bases of a matroid}, Proc. {S}econd
  {L}ouisiana {C}onf. on {C}ombinatorics, {G}raph {T}heory and {C}omputing
  ({L}ouisiana {S}tate {U}niv., {B}aton {R}ouge, {L}a., 1971), 1971,
  pp.~387--410. \MR{0422061}

\bibitem[Oxl91]{oxley-paving}
James~G. Oxley, \emph{Ternary paving matroids}, Discrete Math. \textbf{91}
  (1991), no.~1, 77--86. \MR{1120889}

\bibitem[Oxl11]{oxley}
\bysame, \emph{Matroid theory}, second ed., Oxford Graduate Texts in
  Mathematics, vol.~21, Oxford University Press, Oxford, 2011. \MR{2849819}

\bibitem[Pos09]{postnikov}
Alexander Postnikov, \emph{Permutohedra, associahedra, and beyond}, Int. Math.
  Res. Not. IMRN (2009), no.~6, 1026--1106. \MR{2487491}

\bibitem[PRW08]{postnikov-reiner-williams}
Alexander Postnikov, Victor Reiner, and Lauren Williams, \emph{Faces of
  generalized permutohedra}, Doc. Math. \textbf{13} (2008), 207--273.
  \MR{2520477}

\bibitem[PvdP15]{pendavingh}
Rudi Pendavingh and Jorn van~der Pol, \emph{On the number of matroids compared
  to the number of sparse paving matroids}, Electron. J. Combin. \textbf{22}
  (2015), no.~2, Paper 2.51, 17.

\bibitem[Sch03]{schrijver}
Alexander Schrijver, \emph{Combinatorial optimization. {P}olyhedra and
  efficiency. {V}ol. {B}}, Algorithms and Combinatorics, vol.~24,
  Springer-Verlag, Berlin, 2003, Matroids, trees, stable sets, Chapters 39--69.
  \MR{1956925}

\bibitem[Sib88]{sibuya}
Masaaki Sibuya, \emph{Log-concavity of {S}tirling numbers and unimodality of
  {S}tirling distributions}, Ann. Inst. Statist. Math. \textbf{40} (1988),
  no.~4, 693--714. \MR{996694}

\bibitem[Tru82]{truemper}
K.~Truemper, \emph{Alpha-balanced graphs and matrices and {${\rm
  GF}(3)$}-representability of matroids}, J. Combin. Theory Ser. B \textbf{32}
  (1982), no.~2, 112--139. \MR{657681}

\end{thebibliography}

\end{document}